\documentclass[11pt]{article}
\usepackage[T1]{fontenc}
\usepackage{lmodern}
\usepackage{amsmath,amssymb,amsthm,mathtools}
\usepackage[margin=1in]{geometry}
\usepackage[protrusion=true,expansion=false]{microtype}
\usepackage{enumitem}
\usepackage{booktabs}
\usepackage{hyperref}
\usepackage{tikz}
\usepackage{subcaption}
\hypersetup{
	colorlinks=true,
	linkcolor=blue,
	filecolor=magenta,
	urlcolor=cyan,
	citecolor=red,
	pdftitle={Counterexample to the Bougard--Joret Conjecture},
	pdfsubject={Vertex connectivity, independence number, and Turan extremal graphs},
	pdfauthor={Joyentanuj Das}
}
\newtheorem{theorem}{Theorem}[section]
\newtheorem{proposition}[theorem]{Proposition}
\newtheorem{corollary}[theorem]{Corollary}
\theoremstyle{definition}

\theoremstyle{remark}

\numberwithin{equation}{section}

\newcommand{\e}{\mathrm e}
\newcommand{\ind}{\alpha}
\newcommand{\conn}{\kappa}
\newcommand{\join}{\mathbin{\vee}}

\title{Counterexample to the Bougard--Joret Conjecture}
\author{Joyentanuj Das\thanks{Department of Mathematics, College of Engineering and
		Technology, SRM Institute of Science and Technology, Kattankulathur, Chennai 603203,
		India. E-mail addresses: joyentanuj@gmail.com; joyentad@srmist.edu.in.} \quad and \quad Sayan Gupta\footnote{School of Mathematical Sciences, NISER Bhubaneswar (An OCC of Homi Bhabha National Institute, Mumbai, 400094), India. Email: sayan.gupta@niser.ac.in}}
\date{}

\begin{document}
\maketitle

\begin{abstract}
For admissible integers $n,\alpha,k$, let $f(n,\alpha,k)$ be the minimum number of edges in a $k$-connected graph of order $n$ and independence number $\alpha$. A conjecture of Bougard and Joret predicts that $f(n,\alpha,k)=\lceil nk/2\rceil$ when $n\leq k\alpha$, under the assumptions $n\geq2\alpha$, $n\geq\alpha+k$, $\alpha\geq2$, and $k\geq3$.  We disprove this prediction, determine $f(n,\alpha,k)$ throughout the boundary $n=\alpha+k$, and characterize every extremal graph on that boundary.  In
particular, for every $k\geq4$,
\[
 f(2k-1,k-1,k)=k^2-1,
\]
whereas the conjectured value is $k^2-\lfloor k/2\rfloor$.  The extremal graphs in this family are precisely $\overline K_{k-1}\join T$, where $T$ is an arbitrary tree of order $k$.  The smallest-order failure has parameters $(n,\alpha,k)=(7,3,4)$, and no admissible counterexample has smaller order.
\end{abstract}



\section{Introduction and literature}

All graphs in this paper are finite, simple, and undirected.  The problem of
minimizing the size of a graph subject to a prescribed independence number is
a complementary formulation of the classical extremal problem of Tur\'an
\cite{Turan1941}.  If connectivity is not required, the unique minimizer of
order $n$ and independence number $\alpha$ is the disjoint union of $\alpha$ cliques whose orders differ by at most one.  Ore asked for the corresponding minimum under connectedness \cite{Ore1962}.  The connected problem was solved independently by Christophe et al.\ \cite{ChristopheEtAl2008} and Gitler and Valencia \cite{GitlerValencia2010}; see also the short arguments of Bougard and Joret \cite{BougardJoret2008} and Yuan \cite{Yuan2019}.

For parameters for which the class is nonempty, let $f(n,\alpha,k)$ denote the
minimum size of a $k$-connected graph of order $n$ and independence number
$\alpha$.  A graph attaining this minimum is called
$(n,\alpha,k)$-extremal, or simply extremal when the parameters are clear.

Bougard and Joret determined the $2$-connected case and then considered higher
vertex connectivity \cite{BougardJoret2008}.  For integers satisfying
\begin{equation}\label{eq:admissible}
 n\geq2\alpha,\qquad n\geq\alpha+k,\qquad \alpha\geq2,
 \qquad k\geq3,
\end{equation}
they conjectured that
\begin{equation}\label{eq:conjecture}
 f(n,\alpha,k)=
 \begin{cases}
 \left\lceil \dfrac{nk}{2}\right\rceil,&n\leq k\alpha,\\[6pt]
 t(n,\alpha)+\left\lceil\dfrac{k\alpha}{2}\right\rceil,
 &n>k\alpha.
 \end{cases}
\end{equation}
This is Conjecture~1 in Section~6 of \cite[pp.~11--12]{BougardJoret2008}.
Here $t(n,\alpha)$ is the size of the balanced union of $\alpha$ cliques such that the total order is $n$. This notation must be distinguished from the convention in which a Tur\'an graph is complete multipartite. The original paper established \eqref{eq:conjecture} on the boundary $n=k\alpha$: for $\alpha=2$ this covers all $n\geq2k$, and for $\alpha\geq3$ it covers the range
\[
 n\geq
 \left\lceil\frac{(k-2)\alpha}{2}\right\rceil\alpha+2\alpha,
\] a bound obtained from a clique-partition theorem \cite[Section~6, p.~12]{BougardJoret2008}. Wang and Wu later recorded that
the conjecture for $k\geq3$ was ``still unsettled'' \cite[Section~1]{WangWu2022}.  Liu and Ning reproduced the statement as their
Conjecture~3.1 and described it as ``still wide open'' \cite[Section~3, p.~121]{LiuNing2026}.  We are not aware of an earlier
counterexample.

The first line of \eqref{eq:conjecture} starts from the universal inequality
$\delta(G)\geq k$, which gives $2\e(G)\geq nk$.  A degree sum, however, does
not detect every connectivity cost.  On the boundary $n=\alpha+k$, the
admissibility condition $n\geq2\alpha$ is equivalent to $k\geq\alpha$; moreover,
$n\leq k\alpha$ holds automatically because $\alpha\geq2$ and $k\geq3$.
Deleting a maximum independent set leaves exactly $k$ vertices.  When
$\alpha=k-1$, those vertices must induce a connected graph.  That residual
connectivity costs $k-1$ edges, whereas the degree-sum estimate for a
one-connected graph accounts for only $\lceil k/2\rceil$ edges.  This is the
source of the counterexamples.

Section~\ref{sec:boundary} determines the exact boundary value and all graphs
attaining it.  It shows that \eqref{eq:conjecture} is correct on
$n=\alpha+k$ except precisely when $k-\alpha=1$ and $k\geq4$.
Section~\ref{sec:consequences} proves that order seven is minimal and records
the standard clique-partition estimate relevant to the second line of
\eqref{eq:conjecture}.  The latter is included to distinguish the unresolved
second regime from the boundary obstruction found here.

\section{Notation and preliminary facts}\label{sec:notation}

For a graph $G$, write $V(G)$ and $E(G)$ for its vertex and edge sets,
$\e(G)=|E(G)|$ for its size, $\delta(G)$ for its minimum degree,
$\ind(G)$ for its independence number, and $\conn(G)$ for its vertex
connectivity.  A graph is $k$-connected if it has more than $k$ vertices and
deleting any set of fewer than $k$ vertices leaves a connected graph.  The
join $G\join H$ is obtained from disjoint copies of $G$ and $H$ by adding every
edge with one endpoint in each copy.  We write $\overline K_s$ for the
edgeless graph of order $s$ and $P_s$ for the path of order $s$.

If $n=q\alpha+r$, where $0\leq r<\alpha$, define
\begin{equation}\label{eq:turan-size}
 t(n,\alpha)
 =r\binom{q+1}{2}+(\alpha-r)\binom q2.
\end{equation}
Thus $t(n,\alpha)$ is the number of edges in
$rK_{q+1}\cup(\alpha-r)K_q$.  Convexity of $x\mapsto\binom{x}{2}$ implies
that among all partitions $s_1+\cdots+s_\alpha=n$ into nonnegative integers,
the quantity $\sum_i\binom{s_i}{2}$ is minimized when the $s_i$ differ by at
most one; its minimum is \eqref{eq:turan-size}.

For integers $p\geq2$ and $0\leq d\leq p-1$, let $m(p,d)$ be the minimum size
of a graph of order $p$ that is $d$-connected, with the convention that no
condition is imposed when $d=0$.

For $d\geq2$, the next exact minimum follows from Harary's determination of
the maximum vertex connectivity at fixed order and size
\cite{Harary1962}; the cases $d=0,1$ are elementary.  We record the formula
for later use, especially its exceptional line for $d=1$.

\begin{proposition}\label{prop:min-connectivity-size}
For $p\geq2$,
\begin{equation}\label{eq:m-p-d}
 m(p,d)=
 \begin{cases}
 0,&d=0,\\
 p-1,&d=1,\\
 \left\lceil\dfrac{pd}{2}\right\rceil,&2\leq d\leq p-1.
 \end{cases}
\end{equation}
\end{proposition}

The exceptional second line of \eqref{eq:m-p-d} is essential below.  For
$d=1$, the degree-sum quantity $\lceil p/2\rceil$ equals $p-1$ only when
$p\in\{2,3\}$ and is strictly smaller when $p\geq4$.

\section{The exact value at the order boundary}\label{sec:boundary}

On $n=\alpha+k$, the assumptions in \eqref{eq:admissible} are equivalent to
$\alpha\geq2$, $k\geq3$, and $k\geq\alpha$.  Every such triple belongs to the
first regime of \eqref{eq:conjecture}.

\begin{theorem}\label{thm:boundary}
Let $\alpha\geq2$, $k\geq3$, and $k\geq\alpha$.  Then
\begin{equation}\label{eq:boundary-general}
 f(\alpha+k,\alpha,k)=\alpha k+m(k,k-\alpha).
\end{equation}
Equivalently,
\begin{equation}\label{eq:boundary-piecewise}
 f(\alpha+k,\alpha,k)=
 \begin{cases}
 k^2,&\alpha=k,\\
 k^2-1,&\alpha=k-1,\\
 \left\lceil\dfrac{k(\alpha+k)}2\right\rceil,
 &2\leq\alpha\leq k-2.
 \end{cases}
\end{equation}
\end{theorem}

\begin{proof}
Set $n=\alpha+k$ and put
\[
 d=k-\alpha.
\]
The assumption $k\geq\alpha$ gives $d\geq0$, while $\alpha\geq2$ gives
$d<k$.  We prove a lower bound valid for every admissible graph and then
construct a graph that attains it.

Let $G$ be a $k$-connected graph of order $n$ with
$\ind(G)=\alpha$.  Choose a maximum independent set $S\subseteq V(G)$ and
write
\[
 T=V(G)\setminus S.
\]
Then $|S|=\alpha$ and
\[
 |T|=n-|S|=(\alpha+k)-\alpha=k.
\]
Because $G$ is $k$-connected, every vertex has degree at least $k$.
Now fix $s\in S$.  Since $S$ is independent, $s$ has no neighbor in $S$.
Thus all neighbors of $s$ must lie in the $k$-element set $T$.  The two
inequalities
\[
 k\leq\deg_G(s)\leq |T|=k
\]
force $\deg_G(s)=k$ and $N_G(s)=T$.  Since this holds for every $s\in S$,
all possible edges between $S$ and $T$ occur.  There are exactly
$|S||T|=\alpha k$ such edges, while there are no edges inside $S$.  If
\[
 H=G[T],
\]
then the edge set of $G$ is consequently the disjoint union of the
$\alpha k$ edges between $S$ and $T$ and the edges of $H$.  Therefore
\begin{equation}\label{eq:boundary-decomposition}
 \e(G)=\alpha k+\e(H).
\end{equation}

We next determine the connectivity that $G$ forces on $H$.  If $d=0$, our
definition of $m(k,0)$ imposes no condition, so nothing is required.  Suppose
that $d\geq1$, and let $X\subseteq T$ be arbitrary with $|X|<d$.  The sets
$S$ and $X$ are disjoint, and hence
\[
 |S\cup X|=\alpha+|X|<\alpha+d
                  =\alpha+(k-\alpha)=k.
\]
Deleting $S\cup X$ from $G$ therefore leaves a connected graph.  But the
vertices that remain are precisely $T\setminus X$, and all remaining edges
are the edges of $H-X$; in other words,
\[
 G-(S\cup X)=H-X.
\]
Thus $H-X$ is connected for every $X$ with $|X|<d$.  Moreover,
$|V(H)|=k>d$, so the order requirement in the definition of
$d$-connectivity is satisfied.  It follows that $H$ is $d$-connected.
By the definition of $m(k,d)$ and \eqref{eq:boundary-decomposition}, every
admissible $G$ satisfies
\begin{equation}\label{eq:boundary-lower}
 \e(G)=\alpha k+\e(H)\geq\alpha k+m(k,d).
\end{equation}

We now show that equality is attainable.  Choose a graph $H_d$ of order $k$
with
\[
 \e(H_d)=m(k,d)
\]
that is $d$-connected when $d\geq1$.  Such a graph exists by
Proposition~\ref{prop:min-connectivity-size}; more explicitly, we may take
$H_0=\overline K_k$ when $d=0$, $H_1=P_k$ when $d=1$, and any graph whose
existence is guaranteed by Proposition~\ref{prop:min-connectivity-size} when
$d\geq2$.  On a disjoint set $A$ of $\alpha$ vertices, form
\begin{equation}\label{eq:boundary-construction}
 G_0=\overline K_\alpha\join H_d.
\end{equation}
Thus $A$ is independent, every vertex of $A$ is adjacent to every vertex of
$H_d$, and the only other edges are those of $H_d$.  It follows immediately
that
\begin{equation}\label{eq:boundary-upper-size}
 \e(G_0)=\alpha k+\e(H_d)=\alpha k+m(k,d).
\end{equation}
We must verify that $G_0$ has exactly the prescribed independence number and
is $k$-connected.

No independent set of a join can contain a vertex from each side, because
every vertex on one side is adjacent to every vertex on the other.  Hence
\begin{equation}\label{eq:join-independence}
 \ind(G_0)=\max\{\alpha,\ind(H_d)\}.
\end{equation}
We check $\ind(H_d)\leq\alpha$ in each of the three possible cases.

If $d=0$, then $k=\alpha$ and $H_d=\overline K_k$, so
$\ind(H_d)=k=\alpha$.  In this case $G_0=K_{k,k}$.

If $d=1$, then $k-\alpha=1$, so $\alpha=k-1$.  We chose $H_d=P_k$.  This
path has an edge and therefore its entire $k$-vertex set is not independent;
consequently $\ind(H_d)\leq k-1=\alpha$.

Finally, suppose $d\geq2$.  Since $H_d$ is $d$-connected,
$\delta(H_d)\geq d$.  If $H_d$ contained an independent set $I$ with
$|I|\geq\alpha+1$, choose any $u\in I$.  The vertex $u$ has no neighbor in
$I$, so every neighbor of $u$ belongs to $V(H_d)\setminus I$.  There are at
most
\[
 k-(\alpha+1)=k-\alpha-1=d-1
\]
vertices in that set, giving $\deg_{H_d}(u)\leq d-1$, a contradiction.
Thus $\ind(H_d)\leq\alpha$.  Equation \eqref{eq:join-independence} now gives
\[
 \ind(G_0)=\alpha
\]
in all cases.

It remains to prove $k$-connectivity.  The graph $G_0$ has
$\alpha+k>k$ vertices.  Let $Y\subseteq V(G_0)$ be any set with $|Y|<k$.
Because $H_d$ has exactly $k$ vertices, at least one vertex of $H_d$ survives
the deletion of $Y$.

Suppose first that at least one vertex of the independent side $A$ also
survives.  Then both sides of the join have a surviving vertex.  Every
surviving vertex of $A$ is adjacent to every surviving vertex of $H_d$.
It follows that all surviving vertices lie in a single component: two
vertices on different sides are adjacent, two vertices of $A$ have a common
surviving neighbor in $H_d$, and two vertices of $H_d$ have a common
surviving neighbor in $A$.  Hence $G_0-Y$ is connected, independently of
the internal structure of $H_d-Y$.

Suppose instead that no vertex of $A$ survives.  Then $A\subseteq Y$.  If
$d=0$, then $|A|=\alpha=k$, contradicting $|Y|<k$, so this case can occur
only when $d\geq1$.  Put $X=Y\cap V(H_d)$.  Because $A$ and $V(H_d)$ are
disjoint and $A\subseteq Y$,
\[
 |X|=|Y|-|A|<k-\alpha=d.
\]
All remaining vertices lie in $H_d$, and
\[
 G_0-Y=H_d-X.
\]
The latter graph is connected because $H_d$ is $d$-connected.  We have
therefore shown that deletion of every set of fewer than $k$ vertices leaves
$G_0$ connected.  Thus $G_0$ is $k$-connected.

The graph $G_0$ is admissible, and \eqref{eq:boundary-upper-size} matches the
universal lower bound \eqref{eq:boundary-lower}.  This proves
\eqref{eq:boundary-general}.

It remains only to rewrite the answer.  If $\alpha=k$, then $d=0$, so
\[
 \alpha k+m(k,0)=k^2.
\]
If $\alpha=k-1$, then $d=1$, and Proposition~\ref{prop:min-connectivity-size} gives
\[
 \alpha k+m(k,1)=k(k-1)+(k-1)=k^2-1.
\]
If $2\leq\alpha\leq k-2$, then $d=k-\alpha\geq2$, and
\begin{align*}
 \alpha k+m(k,d)
 &=\alpha k+\left\lceil\frac{k(k-\alpha)}2\right\rceil\\
 &=\left\lceil
       \alpha k+\frac{k(k-\alpha)}2
   \right\rceil\\
 &=\left\lceil\frac{k(\alpha+k)}2\right\rceil.
\end{align*}
In the second equality we used the elementary identity
$z+\lceil x\rceil=\lceil z+x\rceil$ for an integer $z$.  These three cases
give \eqref{eq:boundary-piecewise}.
\end{proof}

\begin{corollary}\label{cor:boundary-characterization}
Let $\alpha\geq2$, $k\geq3$, $k\geq\alpha$, and put $d=k-\alpha$.  A
$k$-connected graph $G$ of order $\alpha+k$ and independence number $\alpha$
is extremal if and only if
\[
 G\cong\overline K_\alpha\join H,
\]
where $H$ has order $k$ and satisfies one of the following conditions:
\begin{enumerate}[label=\textup{(\roman*)}]
\item if $d=0$, then $H\cong\overline K_k$;
\item if $d=1$, then $H$ is a tree;
\item if $d\geq2$, then $H$ is $d$-connected and
      $\e(H)=\lceil kd/2\rceil$.
\end{enumerate}
\end{corollary}

\begin{proof}
We first prove necessity.  Let $G$ be an extremal graph for
$f(\alpha+k,\alpha,k)$.  Choose a maximum independent set
$S\subseteq V(G)$, so $|S|=\alpha$, and put
\[
 T=V(G)\setminus S,
 \qquad H=G[T].
\]
Then $|T|=k$.  As shown in the lower-bound part of
Theorem~\ref{thm:boundary}, $\delta(G)\geq k$, while a vertex of $S$ can have
neighbors only among the $k$ vertices of $T$.  Hence every vertex of $S$ is
adjacent to every vertex of $T$.  There are no edges inside $S$, and
therefore
\begin{equation}\label{eq:cor-characterization-decomposition}
 G\cong\overline K_\alpha\join H,
 \qquad
 \e(G)=\alpha k+\e(H).
\end{equation}

The same proof also shows that $H$ is $d$-connected when $d\geq1$.  Indeed,
if $X\subseteq T$ and $|X|<d$, then
\[
 |S\cup X|=\alpha+|X|<\alpha+d=k,
\]
so $k$-connectivity of $G$ implies that
$G-(S\cup X)=H-X$ is connected.  When $d=0$, no condition on $H$ is imposed.

Since $G$ is extremal, Theorem~\ref{thm:boundary} and
\eqref{eq:cor-characterization-decomposition} give
\[
 \alpha k+\e(H)=\e(G)
   =f(\alpha+k,\alpha,k)
   =\alpha k+m(k,d).
\]
Cancelling $\alpha k$ yields
\begin{equation}\label{eq:cor-characterization-h-size}
 \e(H)=m(k,d).
\end{equation}
We now read this equality in the three possible cases.

If $d=0$, Proposition~\ref{prop:min-connectivity-size} gives
$m(k,0)=0$.  Thus $H$ has no edges and, since it has order $k$,
$H\cong\overline K_k$.  If $d=1$, then $H$ is connected and
\eqref{eq:cor-characterization-h-size} gives $\e(H)=m(k,1)=k-1$.  A
connected graph of order $k$ contains a spanning tree with $k-1$ edges; as
$H$ has no additional edges, it is itself a tree.  Finally, if $d\geq2$,
then $H$ is $d$-connected and Proposition~\ref{prop:min-connectivity-size}
turns \eqref{eq:cor-characterization-h-size} into
\[
 \e(H)=\left\lceil\frac{kd}{2}\right\rceil.
\]
This proves that every extremal graph has one of the stated forms.

We next prove sufficiency.  Let $H$ be a graph of order $k$ satisfying the condition corresponding to $d$, let $A$ be an independent set of order
$\alpha$ disjoint from $H$, and set
\[
 G=\overline K_\alpha\join H.
\]
Since the number of edges between $\overline K_\alpha$ and $H$ is $\alpha k$, the total number of edges $e(G)$ is clearly $\alpha k + m(k,d)$. It remains to show that $\alpha(G)=\alpha$ and $\kappa(G)= k$. However, the proof follows exactly the same way as Theorem~\ref{thm:boundary} and hence we skip it here. Therefore, every graph of one of the stated forms is extremal, completing the
proof of both directions.
\end{proof}

\begin{corollary}\label{cor:counterfamily}
For every integer $k\geq4$,
\begin{equation}\label{eq:counterfamily}
 f(2k-1,k-1,k)=k^2-1.
\end{equation}
The extremal graphs are precisely
$\overline K_{k-1}\join T$, where $T$ is a tree of order $k$.  In particular,
\eqref{eq:conjecture} is false.
\end{corollary}

\begin{proof}
Fix an integer $k\geq4$ and put
\[
 \alpha=k-1,\qquad n=2k-1.
\]
We first verify that all hypotheses of the extremal problem are satisfied.
Certainly $\alpha\geq3\geq2$ and $k\geq3$.  Moreover,
\[
 n=2k-1=(k-1)+k=\alpha+k
\]
and
\[
 n=2k-1\geq2k-2=2\alpha.
\]
Thus the triple $(n,\alpha,k)$ satisfies \eqref{eq:admissible}.  Since
$n=\alpha+k$ and $k\geq\alpha$, Theorem~\ref{thm:boundary} applies.  Here
\[
 d=k-\alpha=k-(k-1)=1,
\]
so either \eqref{eq:boundary-general} and
Proposition~\ref{prop:min-connectivity-size}, or directly the middle line of
\eqref{eq:boundary-piecewise}, gives
\begin{align*}
 f(2k-1,k-1,k)
 &= (k-1)k+m(k,1)\\
 &= k(k-1)+(k-1)\\
 &= k^2-1.
\end{align*}
This proves \eqref{eq:counterfamily}.

We next determine all equality cases.  In the present specialization
$d=1$.  Corollary~\ref{cor:boundary-characterization} says that an extremal
graph is precisely a join
\[
 \overline K_{k-1}\join H,
\]
where $H$ is a tree of order $k$.  Conversely, every such tree gives an
extremal graph by the same corollary.  Thus the displayed family is neither
merely a collection of examples nor dependent on the choice of a particular
tree: it is the complete extremal family.

It remains to compare the exact value with the proposed formula.  The
parameters lie in the first regime of \eqref{eq:conjecture}, because
\begin{align*}
 k\alpha-n
 &=k(k-1)-(2k-1)\\
 &=k^2-3k+1\\
 &=(k-1)(k-2)-1>0
\end{align*}
for every $k\geq4$.  Hence \eqref{eq:conjecture} predicts
\begin{align*}
 \left\lceil\frac{nk}{2}\right\rceil
 &=\left\lceil\frac{k(2k-1)}2\right\rceil\\
 &=\left\lceil k^2-\frac{k}{2}\right\rceil\\
 &=k^2-\left\lfloor\frac{k}{2}\right\rfloor.
\end{align*}
The last equality uses
$\lceil N-x\rceil=N-\lfloor x\rfloor$ when $N$ is an integer.  Since
$k\geq4$ implies $\lfloor k/2\rfloor\geq2$, the conjectured value is at most
$k^2-2$, whereas the exact value is $k^2-1$.  More precisely, the difference
between the exact and conjectured values is
\[
 (k^2-1)-
 \left(k^2-\left\lfloor\frac{k}{2}\right\rfloor\right)
 =\left\lfloor\frac{k}{2}\right\rfloor-1>0.
\]
Therefore \eqref{eq:conjecture} fails for every $k\geq4$.
\end{proof}

\section{Consequences and remaining directions}\label{sec:consequences}

\begin{proposition}\label{prop:smallest}
Among triples satisfying \eqref{eq:admissible}, the minimum possible order of
a counterexample to \eqref{eq:conjecture} is seven.  More precisely,
$f(7,3,4)=15$, while \eqref{eq:conjecture} predicts $14$.
\end{proposition}

\begin{proof}
We first establish the failure at order seven.  For
$(n,\alpha,k)=(7,3,4)$ we have $n=\alpha+k$ and $d=k-\alpha=1$.
Theorem~\ref{thm:boundary} therefore gives
\[
 f(7,3,4)=3\cdot4+m(4,1)=12+3=15.
\]
This value is attained, for example, by
\[
 G_0=\overline K_3\join P_4.
\]
Indeed, $G_0$ has the twelve edges between the two join sides and the three
edges of $P_4$, for a total of fifteen.  An independent set in a join lies
in one side, so
\[
 \ind(G_0)=\max\{3,\ind(P_4)\}=3.
\]
To check connectivity directly, delete at most three vertices.  At least one
of the four path vertices survives.  If a vertex of the independent
three-vertex side also survives, the surviving portions of the two sides are
joined completely and hence form a connected graph.  If the entire
independent side is deleted, no path vertex was deleted, and the remaining
graph is $P_4$, which is connected.  Thus $G_0$ is $4$-connected.  On the
other hand, $7\leq4\cdot3$, so the first line of \eqref{eq:conjecture}
predicts only
\[
 \left\lceil\frac{7\cdot4}{2}\right\rceil=14.
\]
Consequently a counterexample of order seven exists.

There is also a useful direct explanation of why the predicted value
fourteen cannot occur.  Suppose, to the contrary, that a $4$-connected graph
$G$ of order seven with $\ind(G)=3$ has fourteen edges.  Since
$\delta(G)\geq4$,
\[
 28=2\e(G)=\sum_{v\in V(G)}\deg_G(v)\geq7\cdot4=28.
\]
Equality at the two ends forces every vertex to have degree four; hence $G$
is $4$-regular.  Each vertex consequently has degree
$6-4=2$ in the complement $\overline G$, so $\overline G$ is $2$-regular.

The equality $\ind(G)=3$ means that some three vertices are pairwise
nonadjacent in $G$, and these vertices induce a triangle in $\overline G$.
Each vertex of this triangle already has its two required neighbors within
the triangle.  Since $\overline G$ is $2$-regular, no edge of $\overline G$
can join the triangle to any of the other four vertices.  Thus the triangle
is an entire component of $\overline G$.  The remaining four vertices also
induce a simple $2$-regular graph, and the only such graph on four vertices
is a $4$-cycle.  Therefore
\[
 \overline G=C_3\cup C_4.
\]
Delete from $G$ the three vertices belonging to the $C_3$ component of
$\overline G$.  The graph left in $G$ is
\[
 \overline{C_4}=2K_2,
\]
which is disconnected.  This contradicts $4$-connectivity because only
three vertices were deleted.

We now prove that no smaller admissible order can yield a counterexample.
The inequalities in \eqref{eq:admissible} include
\[
 \alpha\geq2,\qquad k\geq3,\qquad
 n\geq2\alpha,\qquad n\geq\alpha+k.
\]
In particular, $n\geq2+3=5$, so there is no admissible triple with
$n\leq4$.  If $n=5$, the inequality $2\alpha\leq5$ and the condition
$\alpha\geq2$ force $\alpha=2$; then
$3\leq k\leq n-\alpha=3$, so $k=3$.  This gives only
\[
 (n,\alpha,k)=(5,2,3).
\]
If $n=6$, then $2\leq\alpha\leq3$.  For $\alpha=2$, the inequalities
$3\leq k\leq6-2$ give $k\in\{3,4\}$; for $\alpha=3$, they give
$3\leq k\leq6-3$, and hence $k=3$.  The complete list of admissible triples
with $n\leq6$ is therefore
\begin{equation}\label{eq:small-admissible-triples}
 (5,2,3),\qquad (6,2,3),\qquad
 (6,2,4),\qquad (6,3,3).
\end{equation}

Every triple in \eqref{eq:small-admissible-triples} satisfies
$n\leq k\alpha$; this also follows at once from
$n\leq6\leq k\alpha$.  Hence the conjectured value in every case is
$\lceil nk/2\rceil$.  Any $k$-connected graph has minimum degree at least
$k$, and the handshaking identity yields
\begin{equation}\label{eq:small-degree-lower}
 \e(G)\geq\left\lceil\frac{nk}{2}\right\rceil.
\end{equation}
It remains to exhibit, for each triple, a graph attaining this lower bound
with the required independence number and connectivity.

For $(5,2,3)$, let $G_1=K_5-E(M_2),$ where $M_2$ is a matching of two edges (for reference see Figure~\ref{fig:sub1}). It has
\[
 \e(G_1)=\binom52-2=8
        =\left\lceil\frac{5\cdot3}{2}\right\rceil.
\]

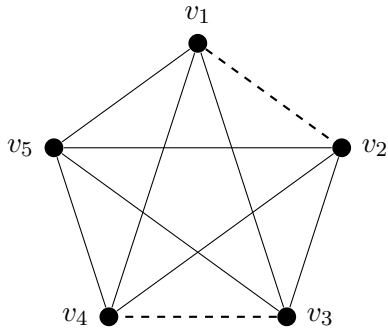
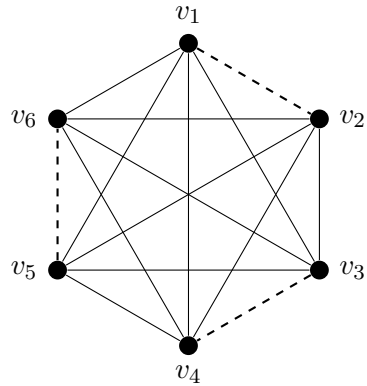
\begin{figure}[htbp]
    \centering
    \begin{subfigure}[b]{0.48\textwidth}
        \centering
        \begin{tikzpicture}[
            vertex/.style={circle, fill=black, inner sep=2.5pt},
            deleted/.style={dashed, black, thick}
        ]
            \useasboundingbox (-2.5,-3.2) rectangle (2.5,3.2);

            \node[vertex, label=above:$v_1$] (v1) at (90:2) {};
            \node[vertex, label=right:$v_2$] (v2) at (18:2) {};
            \node[vertex, label=right:$v_3$] (v3) at (-54:2) {};
            \node[vertex, label=left:$v_4$]  (v4) at (-126:2) {};
            \node[vertex, label=left:$v_5$]  (v5) at (162:2) {};

            \draw (v1) -- (v3); \draw (v1) -- (v4); \draw (v1) -- (v5);
            \draw (v2) -- (v3); \draw (v2) -- (v4); \draw (v2) -- (v5);
            \draw (v3) -- (v5); \draw (v4) -- (v5);

            \draw[deleted] (v1) -- (v2);
            \draw[deleted] (v3) -- (v4);

        \end{tikzpicture}
        \caption{$K_5 - E(M_2)$}
        \label{fig:sub1}
    \end{subfigure}
    \hfill
    \begin{subfigure}[b]{0.48\textwidth}
        \centering
        \begin{tikzpicture}[
            vertex/.style={circle, fill=black, inner sep=2.5pt},
            deleted/.style={dashed, black, thick}
        ]
            \useasboundingbox (-2.5,-3.2) rectangle (2.5,3.2);

            \node[vertex, label=above:$v_1$] (v1) at (90:2) {};
            \node[vertex, label=right:$v_2$] (v2) at (30:2) {};
            \node[vertex, label=right:$v_3$] (v3) at (-30:2) {};
            \node[vertex, label=below:$v_4$] (v4) at (-90:2) {};
            \node[vertex, label=left:$v_5$] (v5) at (-150:2) {};
            \node[vertex, label=left:$v_6$] (v6) at (150:2) {};

            \draw (v1) -- (v3); \draw (v1) -- (v4); \draw (v1) -- (v5); \draw (v1) -- (v6);
            \draw (v2) -- (v3); \draw (v2) -- (v4); \draw (v2) -- (v5); \draw (v2) -- (v6);
            \draw (v3) -- (v5); \draw (v3) -- (v6);
            \draw (v4) -- (v5); \draw (v4) -- (v6);

            \draw[deleted] (v1) -- (v2);
            \draw[deleted] (v3) -- (v4);
            \draw[deleted] (v5) -- (v6);

        \end{tikzpicture}
        \caption{$K_6 - E(M_3)$}
        \label{fig:sub2}
    \end{subfigure}
    
    \caption{Graphs $G_1$ and $G_2$.}
    \label{fig:main}
\end{figure}

The endpoints of either deleted edge form an independent pair.  No three vertices are independent, because the non-edges of $G_1$ form a matching and therefore cannot contain all three pairs among a three-element set.  Thus $\ind(G_1)=2$.  After deleting fewer than three vertices, at least three vertices remain.  The remaining graph is a complete graph with some pairwise disjoint edges removed.  If two surviving vertices are nonadjacent, they are the endpoints of one removed matching edge, and every third surviving vertex is adjacent to both.  Thus all surviving vertices lie in one component.  Hence $G_1$ is $3$-connected.

For $(6,2,4)$, let $G_2=K_6-E(M_3),$ where $M_3$ is a perfect matching (for reference see Figure~\ref{fig:sub2}). Then
\[
 \e(G_2)=\binom62-3=12=\frac{6\cdot4}{2}.
\]
Exactly the same matching argument gives $\ind(G_2)=2$.  Deleting fewer than four vertices leaves at least three vertices, and the same common-neighbor argument shows that the remaining graph is connected. Therefore $G_2$ is $4$-connected.

For $(6,2,3)$, we use the triangular prism $G_3$ (for reference see Figure~\ref{fig:G3}).  Denote its two vertex-disjoint triangles by
\[
 A=\{a_1,a_2,a_3\},\qquad B=\{b_1,b_2,b_3\},
\]
and let its remaining three edges be $a_ib_i$ for $1\leq i\leq3$.  The graph has three edges in each triangle and three matching edges, so it has nine edges.  An independent set contains at most one vertex of each triangle, while $a_1$ and $b_2$ are nonadjacent; hence its independence number is exactly two.  Delete at most two vertices.  Each triangle retains at least one vertex, and two deleted vertices cannot meet all three pairwise vertex-disjoint matching edges.  Thus some matching edge $a_ib_i$ survives.  Each nonempty remnant of a triangle is connected, and the surviving matching edge joins the two remnants.  The resulting graph is connected, so the triangular prism is $3$-connected.

\begin{figure}[ht]
\centering
\begin{tikzpicture}[
    vertex/.style={circle, fill=black, inner sep=2.5pt}]

\node[vertex, label=left:$a_1$]  (a1) at (0,0) {};
\node[vertex, label=right:$a_2$] (a2) at (3,0) {};
\node[vertex, label=above:$a_3$] (a3) at (1.5,2.6) {};

\node[vertex, label=left:$b_1$]  (b1) at (1.2,0.8) {};
\node[vertex, label=right:$b_2$] (b2) at (4,0.8) {};
\node[vertex, label=above:$b_3$] (b3) at (2.5,3.4) {};

\draw (a1) -- (a2);
\draw (a2) -- (a3);
\draw (a3) -- (a1);

\draw (b1) -- (b2);
\draw (b2) -- (b3);
\draw (b3) -- (b1);

\draw (a1) -- (b1);
\draw (a2) -- (b2);
\draw (a3) -- (b3);

\end{tikzpicture}
    \caption{$G_3$}
    \label{fig:G3}
\end{figure}

Finally, for $(6,3,3)$, consider the graph $G_4=K_{3,3}$. It has nine edges.  Each bipartition class is an independent set of order three, while an independent set cannot meet both classes because every possible cross-edge is present.  Thus $\ind(G_4)=3$.  Deleting fewer than three vertices cannot remove all three vertices of either bipartition class.  At least one vertex therefore survives in each class, and the remaining graph is a complete bipartite graph with both parts nonempty, which is connected.  Hence $G_4$ is $3$-connected.

All four graphs attain the corresponding lower bound \eqref{eq:small-degree-lower}.  Thus \eqref{eq:conjecture} holds for every admissible triple of order at most six.  Since it fails for $(7,3,4)$, seven is the minimum possible order of a counterexample.
\end{proof}

The following standard clique-partition estimate is the counting implication
used by Bougard and Joret in their verification of the large-order range
\cite[Section~6, p.~12]{BougardJoret2008}.  We record it for reference; it
does not imply that every extremal graph has the required partition.

\begin{proposition}\label{prop:partition}
Let $\alpha\geq2$, and let $G$ be a $k$-connected graph of order $n$.  If
$V(G)$ has a partition $C_1\cup\cdots\cup C_\alpha$ into $\alpha$ nonempty
cliques, then
\begin{equation}\label{eq:partition-bound}
 \e(G)\geq t(n,\alpha)+\left\lceil\frac{k\alpha}{2}\right\rceil.
\end{equation}
\end{proposition}

Bougard and Joret proved a clique-partition conclusion under an explicit
large-order hypothesis and thereby verified the second line of
\eqref{eq:conjecture} in that range \cite{BougardJoret2008}.  The general
second regime remains a separate question: a proof must either establish the
partition or obtain \eqref{eq:partition-bound} without it.  The counterfamily
above shows that the first regime requires revision before a unified statement
can be true.  The boundary $n=\alpha+k$ is completely determined by
\eqref{eq:boundary-piecewise}, and Corollary~\ref{cor:boundary-characterization}
determines every equality case there.  Any corrected conjecture must therefore
include the exceptional value $k^2-1$ and the extremal family
$\overline K_{k-1}\join T$ when $\alpha=k-1$.

\section*{Declaration of generative AI and AI-assisted technologies in the manuscript preparation process}
During the preparation of this work, the authors used AI to discuss proof strategies, organize and check bibliographic information, check algebraic calculations and proof exposition, and improve language and readability.  After using AI, the authors reviewed and edited the content as needed, independently verified the cited sources, calculations, statements, and proofs, and take full responsibility for the content of the article.

\section*{Declaration of competing interest}
The authors declare that they have no known competing financial interests or personal relationships that could have appeared to influence the work reported in this article.


\section*{Data availability}

Data sharing is not applicable to this article as no datasets were generated or analyzed during the current study.

\end{document}